\documentclass[12pt,a4paper]{article}
\usepackage{geometry}
\usepackage{amsmath,amssymb}
\usepackage{amsfonts}
\usepackage{amsthm}
\usepackage{mathrsfs}
\numberwithin{equation}{section}
\usepackage{comment}
\usepackage{showlabels}
\newcommand{\jap}[1]{\langle #1 \rangle}
\def\a{\alpha}
\def\b{\beta}
\def\c{\gamma}
\def\d{\delta}
\def\e{\varepsilon}
\def\f{\varphi}

\def\l{\lambda}
\def\m{\mu}
\def\n{\nu}

\def\s{\sigma}

\def\x{\xi}
\def\y{\eta}

\def\Fd{\mathcal{F}_d}

\def\re{\mathbb{R}}

\def\ze{\mathbb{Z}}

\def\T{\mathbb{T}}
\def\pa{\partial}

\renewcommand{\Im}{\text{{\rm Im}\;}}

\newcommand{\supp}{\text{{\rm supp}\;}}

\newcommand{\sgn}{\mathrm{sgn}}
\newtheorem{thm}{Theorem}[section]
\newtheorem{lem}[thm]{Lemma}
\newtheorem{prop}[thm]{Proposition}

\theoremstyle{definition}

\theoremstyle{remark}
\newtheorem{rem}[thm]{Remark}

\title{Uniform resolvent estimates for the discrete Schr\"odinger operator in dimension three}

\author{Kouichi Taira, \thanks{Research Organization of Science and Technology, Ritsumeikan university, e-mail:20v00029@gst.ritsumei.ac.jp}}

\date{}

\begin{document}

\maketitle

\begin{abstract}
In this note, we prove the uniform resolvent estimate of the discrete Schr\"odinger operator with dimension three. To do this, we show a Fourier decay of the surface measure on the Fermi surface.
\end{abstract}

\section{Introduction}

We consider the three-dimensional discrete Laplacian
\begin{align*}
H_0 u(x)=-\sum_{|x-y|=1}(u(y)-u(x)).
\end{align*} 
We denote the Fourier expansion by $\Fd$:
\begin{align*}
\hat{u}(\x)=\mathcal{F} u(\x)=\sum_{x\in \ze^3}e^{-2\pi i x\cdot\x}u(x),\quad \x\in \mathbb{T}^3=\re^3/\ze^3.
\end{align*}
Then it follows that
\begin{align}\label{multh_0}
\Fd H_0 u(\x)=h_0(\x)\Fd u(\x), \,\, h_0(\x)=4\sum_{j=1}^3\sin^{2}(\pi \x_j).
\end{align}
We denote the set of the critical points of $h_0$ by $\mathrm{Cr}(h_0)$:
\begin{align}\label{crih_0}
\mathrm{Cr}(h_0)=\{\x\in \T^3\mid \nabla h_0(\x)=0\}=\{\x\in \T^3\mid \x_j\in \{0, 1/2\},\, j=1,2,3\}.
\end{align}
We call $\x\in \mathrm{Cr}(h_0)$ an elliptic threshold if $\x$ attains maximum or minimum of $h_0$ and a hyperbolic threshold otherwise. We set $M_{\l}=h_0^{-1}(\{\l\})$ for $\l\in[0,12]$. The set $M_{\l}$ is called the Fermi surface.

In this note, we show the uniform resolvent estimates for discrete Schr\"odinger operator with dimension three. In case of the continuous Laplacian $-\Delta$ on $\re^d$, the following uniform resolvent estimates are known (\cite{KRS}, \cite{KY}):
\begin{align}\label{unifeuc}
\| (-\Delta-z)^{-1}f\|_{L^q(\re^d)}\leq C|z|^{\frac{d}{2}(\frac{1}{p}-\frac{1}{q})-1}  \|f\|_{L^p(\re^d)}\quad \text{for}\quad z\in \mathbb{C}\setminus [0,\infty),
\end{align}
where $d\geq 3$ and
\begin{align*}
\frac{2}{d+1}\leq \frac{1}{p}-\frac{1}{q}\leq \frac{2}{d},\,\, \frac{2d}{d+3}<\frac{1}{p}<\frac{2d}{d+1},\,\, \frac{2d}{d-1}<\frac{1}{q}<\frac{2d}{d-3}.
\end{align*}
Moreover, it turns out that $(-\Delta-z)^{-1}$ is uniformly bounded in $B(L^p(\re^d), L^{p'}(\re^d))$ with respect to $z\in \mathbb{C}\setminus [0,\infty)$ if and only if $p=2d/(d+2)$. On the other hand, in \cite[Theorem 1.7 (iii)]{TT} (see also Lemma \ref{unifequi}), it is shown that the resolvent $R_0(z)=(H_0-z)^{-1}$ for the discrete Schr\"odinger operator is not bounded from $l^p(\ze^d)$ to $l^{p'}(\ze^d)$ with $p=2d/(d+2)$, $p'=p/(p-1)$ and with $d\geq 5$.
The result in \cite[Proposition 3.3]{TT} shows that the resolvent $R_0(z)$ satisfies
\begin{align}\label{disuniin}
\| R_0(z)f\|_{l^{p'}(\ze^d)}\leq C \|f\|_{l^p(\ze^d)}\quad \text{for}\quad z\in \mathbb{C}\setminus [0,4d],\,\, 1\leq p\leq \frac{2d}{d+3},\,\, d\geq 4
\end{align}
The natural questions are the following:
\begin{itemize}
\item Is the estimate $(\ref{disuniin})$ optimal?
\item What about  the case of $d=3$?
\end{itemize}
For the latter, the authors in \cite{KM} showed that $(\ref{disuniin})$ hold for $p\in [1,\frac{12}{11})$ and for $d=3$ (see also Lemma \ref{unifequi}). In this paper, we improve their results and give the resolvent estimates which is sharp away from the threshold energies.

The proof of $(\ref{disuniin})$ in \cite{TT} depends on the endpoint Strichartz estimates (\cite{SK}).
We point out that the endpoint Strichartz estimates for discrete Schr\"odinger operators might not be used for the sharp resolvent estimate with dimension three since the Strichartz estimates in \cite{SK} are sharp. This is different from the case of the continuous Laplacian $-\Delta$ (in this case, the endpoint Strichartz estimates implies the sharp resolvent estimate $(\ref{unifeuc})$ with $p=2d/d+2$ and with $q=p'$). Instead, we use the strategy in \cite{C2} and calculate the Fourier decay of the surface measure for the Fermi surface. In \cite{ES}, the Fourier decay away from the umbilic points (the points where all principal curvatures vanish) are studied. In this paper, we improve this result and also deal with the Fourier decay near the umbilic point. For its application to the random Schr\"odinger operators, see \cite{ES} and references therein.

The main result of this paper is the following theorem. 

\begin{thm}\label{mainthm} $(i)$ $($Resolvent estimates away from the thresholds$)$ Let $p\in [1,\frac{5}{4}]$ and $r\in [1,\frac{10}{3}]$.
For $\e>0$, set 
\begin{align*}
D_{\e}=\bigcap_{k=0}^3\{z\in \mathbb{C}\mid |z-4k|\geq \e\}.
\end{align*}
Then the resolvent $R_0(z)=(H_0-z)^{-1}$ satisfies
\begin{align*}
\sup_{z\in D_{\e}\setminus \re}\|R_0(z)\|_{B(l^p(\ze^3), l^{p'}(\ze^3))}<\infty,\quad \sup_{z\in D_{\e}\setminus \re,\,\, \|W_j\|_{l^r(\ze^3)}=1}\|W_1R_0(z)W_2\|_{B(l^2(\ze^3))}<\infty.
\end{align*}

\noindent $(ii)$ $($Resolvent estimates near the thresholds$)$ Let $p\in [1,\frac{6}{5}]$ and $r\in [1,3]$.
Then we have
\begin{align*}
\sup_{z\in \mathbb{C}\setminus \re}\|R_0(z)\|_{B(l^p(\ze^3), l^{p'}(\ze^3))}<\infty,\quad \sup_{z\in \mathbb{C}\setminus \re,\,\, \|W_j\|_{l^r(\ze^3)}=1}\|W_1R_0(z)W_2\|_{B(l^2(\ze^3))}<\infty.
\end{align*}

\end{thm}

\begin{rem}
\cite[Theorem 1.11 (iii)]{TT} shows that the range of $p$ and $r$ in $(i)$ are optimal.
\end{rem}

\begin{rem}
The above results are proved in \cite{KM} for $p\in [1,\frac{12}{11})$ and $r\in [1, \frac{12}{5})$ (see Lemma \ref{unifequi}).
\end{rem}

\begin{rem}
More generally, it follows from \cite[Theorem 1.2 (i)]{T} that the uniform resolvent estimates away form the diagonal line hold, that is,
\begin{align*} 
\sup_{z\in D_{\e}\setminus \re}\|R_0(z)\|_{B(l^p(\ze^3), l^{q}(\ze^3))}<\infty,
\end{align*}
for
\begin{align*}
\frac{3}{5}\leq \frac{1}{p}-\frac{1}{q},\quad \frac{5}{7}<\frac{1}{p},\quad \frac{1}{q}<\frac{2}{7}.
\end{align*}
Moreover, \cite[Theorem 1.2 (ii)]{T} implies that $R_0(z)$ is H\"older continuous on $B(l^p(\ze^3), l^{p'}(\ze^3))$ for $1\leq p<5/4$. From this result and the proof of \cite[Theorem 1.9]{TT}, it is expected that the wave operators $W_{\pm}=\lim_{t\to \pm \infty}e^{itH}e^{-itH_0}$ exist and is complete for $H=H_0+V$ with $V\in l^{\frac{5}{3}}(\ze^3)$. We omit the detail.
\end{rem}

Finally, we state a possible conjecture on the resolvent estimates near the threshold energies. The author expects that for $p\in [1,\frac{5}{4}]$, the following estimates hold:
\begin{align}\label{conjuni}
\|R_0(z)f\|_{l^{p'}(\ze^3)}\leq C\left(\prod_{k=0}^3|z-4k|^{\frac{3}{2}(\frac{1}{p}-\frac{1}{p'})-1} \right)\|f\|_{l^p(\ze^3)}, \quad \text{for}\quad z\in \mathbb{C}\setminus \re.
\end{align}
By virtue of Proposition \ref{propthre} and \cite[Proposition A.11]{C2}, in order to prove $(\ref{conjuni})$, we only need to prove
\begin{align*}
\|\chi(D)R_0(z)f\|_{l^{p'}(\ze^3)}\leq C\left(\prod_{k=0}^3|z-4k|^{\frac{3}{2}(\frac{1}{p}-\frac{1}{p'})-1} \right)\|f\|_{l^p(\ze^3)}, \quad \text{for}\quad z\in \mathbb{C},
\end{align*}
where $\chi\in C^{\infty}(\mathbb{T}^3)$ is supported around $\x_0$ with $\x_0\in (M_{4}\cup M_{8})\setminus \mathrm{Cr}(h_0)$. The estimates $(\ref{conjuni})$ can be applied with the Keller type eigenvalue bounds for three dimensional discrete Schr\"odinger operators with complex potentials (see \cite{F} for the continuous Laplacian).

%We fix some notations. 

%%%%%%%%%%%%%%%%%%%%%%%%%%%%%%%%%%%%%%%%%%%%%%%%%%%%%%%%%%%%%%%%%%%%%%%%%%%%%%%%%%%%%%%%%%%%%%%%%%%%%%%%%%%%%%%%%%%%%%%%%%%%%%%%%%%%%%%%%%%%%%%%%%%%%%%%%%%%%%%%%%%%%%%%%%%%%%%%%%%%%%%%%%%%%%%%%%%%%%%%%%%%%%%%%%%%%%%%%%%%%%%%%%%%%%%%%%%%%%%%%%%%%%%%%%%%%%%%%%%%%%%%%%%%%%%%%%%%%%%%%%%%%%%%%%%%%%%%%%%%%%%%%%%%%%%%%%%%%%%%%%%%%%%%%%%%%%%%%%%%%%%%%%%%%%%%%%%%%%%%%%%%%%%%%%%%%%%%%%%%%%%%%%%%%%%%%%%%%%%%%%%%%%%%%%%%%%%%%%%%%%%%%%%%%%%%%%%%%%%%%%%%%%%%%%%%%%%%%%%%%%%%%%%%%%%%%%%%%%%%%%%%%%%%%%%%%%%%%%%%%%%%%%%%%%%%%%%%%%%%%%%%%%%%%%%%%%%%%%%%%%%%%%%%%%%%%%%%%%%%%%%%%%%%%%%%

\begin{comment}
For $1\leq p<\infty$, we denote the Shatten spaces by
\begin{align*}
\mathfrak{S}^p=\{T\in B_{\infty}(l^2(\ze^d))\mid \mathrm{Tr}(T^*T)^{\frac{p}{2}}<\infty\},
\end{align*}
where $B_{\infty}(l^2(\ze^d))$ is the set of compact operators on $l^2(\ze^d)$.

\end{comment}

%%%%%%%%%%%%%%%%%%%%%%%%%%%%%%%%%%%%%%%%%%%%%%%%%%%%%%%%%%%%%%%%%%%%%%%%%%%%%%%%%%%%%%%%%%%%%%%%%%%%%%%%%%%%%%%%%%%%%%%%%%%%%%%%%%%%%%%%%%%%%%%%%%%%%%%%%%%%%%%%%%%%%%%%%%%%%%%%%%%%%%%%%%%%%%%%%%%%%%%%%%%%%%%%%%%%%%%%%%%%%%%%%%%%%%%%%%%%%%%%%%%%%%%%%%%%%%%%%%%%%%%%%%%%%%%%%%%%%%%%%%%%%%%%%%%%%%%%%%%%%%%%%%%%%%%%%%%%%%%%%%%%%%%%%%%%%%%%%%%%%%%%%%%%%%%%%%%%%%%%%%%%%%%%%%%%%%%%%%%%%%%%%%%%%%%%%%%%%%%%%%%%%%%%%%%%%%%%%%%%%%%%%%%%%%%%%%%%%%%%%%%%%%%%%%%%%%%%%%%%%%%%%%%%%%%%%%%%%%%%%%%%%%%%%%%%%%%%%%%%%%%%%%%%%%%%%%%%%%%%%%%%%%%%%%%%%%%%%%%%%%%%%%%%%%%%%%%%%%%%%%%%%%%%%%%%%%%%%%%%%%%%%%%%%%%%%%%%%%

\noindent
\textbf{Acknowledgment.}  
This work was partially supported by JSPS Research Fellowship for Young Scientists, KAKENHI Grant Number 17J04478 and 20J00221. Moreover, it is also supported by the program FMSP at the Graduate School of Mathematics Sciences, the University of Tokyo. The author would like to thank Kenichi Ito and Shu Nakamura for encouraging to write this paper. The author is grateful to J.C. Cuenin and I.A. Ikromov for pointing out a mistake of the earlier version of this manuiscript.

\section{Preliminary, reduction to the Fourier decay of the surface measure}

\subsection{Uniform resolvent estimates near thresholds}

To obtain uniform resolvent estimates near thresholds, we only need to the argument in \cite[Proposition 3.3]{TT} slightly.

\begin{prop}\label{propthre}
Let $d\geq 3$ and $T:\mathbb{T}^d\to \re$ be a smooth function with a non-degenerate critical point $\x_0$ with corresponding energy $\l_0$. Then there exists $\d>0$ such that for $\chi\in C_c^{\infty}(B_{\d}(\x_0))$ and for $r\in [1,d]$, we have
\begin{align}\label{Threresol}
\|W_1\chi(D)^2(T(D)-z)^{-1}W_2\|_{B(l^2(\ze^d))}\leq C\|W_1\|_{l^{r}(\ze^d)}\|W_2\|_{l^{r}(\ze^d)}.
\end{align}
with a constant independent of $W_1, W_2 \in l^{r}(\ze^d)$ and $z\in \mathbb{C}\setminus \re$.
\end{prop}

\begin{rem}
In \cite[Proposition A.11]{C2}, it is proved that
\begin{align*}
\|W_1\chi(D)^2(T(D)-z)^{-1}W_2\|_{B(L^2)}\leq C|z-\l_0|^{\frac{d}{r}-1}G_r(z)\|W_1\|_{L^{r}}\|W_2\|_{L^{r}},
\end{align*}
for $r\in [d,d+1]$, where 
\begin{align*}
G_r(z)=\begin{cases}
|\log|z-\l_0||  \quad  \text{if $\x_0$ is a saddle point and if $r=d$}, \\
1\quad \text{otherwise}.
\end{cases}
\end{align*}
Proposition \ref{propthre} improves this result when $\x_0$ is a saddle point and when $r=d$, although in \cite[Proposition A.11]{C2}, the Shatten norm estimates are shown. For the results on exact ultrahyperbolic operators, see \cite{JKL}.
\end{rem}

\begin{proof}
We shall slightly modify the argument in \cite[Proposition 3.3]{TT}. We may assume $\Im z<0$.
Since $l^{p_1}(\ze^d)\subset l^{p_2}(\ze^d)$ for $p_1\leq p_2$, we may assume $r=d$. We take $\d>0$ small enough such that the Hessian of $T(\x)$ does not vanish on $B_{2\d}(\x_0)$. Then for $\chi\in C_c^{\infty}(B_{\d}(\x_0))$, the stationary phase theorem implies
\begin{align*}
\|\chi(D)e^{-itT(D)}\|_{B(l^1(\ze^d), l^{\infty}(\ze^d))}\leq C\jap{t}^{-\frac{d}{2}},
\end{align*}
where we note that the singularity at $t=0$ does not occur by virtue of the compactness of $\supp \chi$ (see the proof in \cite[Theorem 3]{SK}). Applying \cite[Theorem 1.2]{KT} with $U(t)=1_{[0,T)}(t)\chi(D)e^{-itT(D)}$, it follows that the unique solution $u(t,x)$ to
\begin{align}\label{Schro}
i\pa_tu(t,x)-T(D)u(t,x)=g(t,x),\quad u(0,x)=u_0(x)\in l^2(\ze^d)
\end{align}
satisfies
\begin{align*}
\|\chi(D)^2u\|_{L^2([0,T),l^{2^*}(\ze^d) )}\leq C\|u_0\|_{L^2(\ze^d)}+ C\|g\|_{L^2([0,T), l^{2_*}(\ze^d))},
\end{align*}
where $2^*=2d/(d-2)$ and $2_*=2d/(d+2)$. 

Let $f$ be a finitely supported function. Set $g(t,x)=e^{itz}f(x)$, $u_0(x)=(T(D)-z)^{-1}f(x)$ and $u(t,x)=e^{itz}u(x)$. Since $u(t,x)$ and $g(t,x)$ satisfy $(\ref{Schro})$, we have
\begin{align*}
\c(T)\|\chi(D)^2u_0\|_{l^{2^*}(\ze^d) )}\leq C\|u_0\|_{L^2(\ze^d)}+ C\c(T)\|f\|_{l^{2_*}(\ze^d)},
\end{align*}
where $\c(T)=(\int_0^T|e^{it z}|^2 dt)^{1/2}$. Since $\Im z<0$, we have $\c(T)\geq\sqrt{T}$. By letting $T\to\infty$, we obtain 
\begin{align*}
\|\chi(D)^2(T(D)-z)^{-1}f\|_{l^{2^*}(\ze^d) )}\leq  C\|f\|_{l^{2_*}(\ze^d)}.
\end{align*}
Now Lemma \ref{unifequi} implies $(\ref{Threresol})$ for $\Im z<0$.
\end{proof}

\subsection{Uniform resolvent estimates away form thresholds}

We use the following propositions essentially due to the arguments in \cite[Proposition A.5]{C2} and \cite[Theorem 1.2]{T}. Although \cite[Proposition A.5]{C2} is stated only for a hypersurface in $\re^d$, its proof there can be applied with a hypersurface on $\mathbb{T}^d$.

\begin{prop}\label{Fouuni}
Let $d\geq 1$ and $M\subset \T^d$ be a hypersurface with normalized defining function $\rho:\T^d\to \re$. For $\chi\in C^{\infty}(\T^d)$ and $k>0$, assume that
\begin{align}\label{Foudecay}
\sup_{x\in \ze^d}(1+|x|)^k\widehat{\chi d\s_{M}(x)}<\infty,
\end{align}
where $d\s_{M}$ denotes the canonical surface measure on $M$. Then for $r\in [1,2+2k]$, we have
\begin{align*}
\|W_1\chi(D)(\rho(D)-z)^{-1}W_2\|_{B(l^2(\ze^d))}\leq C\|W_1\|_{l^{r}(\ze^d)}\|W_2\|_{l^{r}(\ze^d)}.
\end{align*}
with a constant independent of $W_1, W_2 \in l^{r}(\ze^d)$ and $z\in \mathbb{C}\setminus \re$.
\end{prop}

%\begin{proof}
%This proposition can be proved similar to \cite[Proposition A.5]{C2}. See also \cite[Theorem 1.2]{T}. We note that $\mathfrak{S}^p\subset B(l^2(\ze^d))$ and that $l^{p_1}(\ze^d)\subset l^{p_2}(\ze^d)$ for $p_1\leq p_2$.
%\end{proof}

By using a partition of unity, to prove Theorem \ref{mainthm}, it suffices to prove the following theorem.

\begin{thm}\label{mainpropuni}
Let $\l\in (0,12)$. We denote $M=M_{\l}$ and $\rho(\x)=h_0(\x)-\l$.

\noindent $(i)$ Let $\l\in (0,4)\cup (8,12)$ and $\x\in M_{\l}$. Then for any $\chi\in C^{\infty}(\T^3)$ supported close to $\x$, $(\ref{Foudecay})$ holds for $k=1$.

\noindent $(ii)$ Let $\l=6$ and $\x\in M_{\l}$. Then for any $\chi\in C^{\infty}(\T^3)$ supported close to $\x$, $(\ref{Foudecay})$ holds for $k=\frac{2}{3}$.

\noindent$(iii)$ Let $\l\in (4,8)\setminus \{6\}$ and $\x\in M_{\l}$. Then for any $\chi\in C^{\infty}(\T^3)$ supported close to $\x$, $(\ref{Foudecay})$ holds for $k=\frac{3}{4}$.

\noindent$(iv)$ Let $\l\in \{4,8\}$ and $\x\in M_{\l}\setminus \mathrm{Cr}(h_0)$. Then for any $\chi\in C^{\infty}(\T^3)$ supported close to $\x$, $(\ref{Foudecay})$ holds for $k=\frac{1}{2}$.

\end{thm}

\begin{rem}
In \cite[Theorem 2.1]{ES}, $(iii)$ is proved for $r=\frac{3}{4}-\e$ for any $\e>0$ (more precisely, the estimates with a logarithmic loss). Our result $(iii)$ improves the result in \cite{ES}. 

\end{rem}

\begin{proof}[Proof of Theorem \ref{mainthm}]
Proposition \ref{Fouuni}, Theorem \ref{mainpropuni} $(i)$, $(ii)$, and $(iii)$ imply Theorem \ref{mainthm} $(i)$. Moreover, Proposition \ref{propthre} and Theorem \ref{mainpropuni} imply Theorem \ref{mainthm} $(ii)$.

\end{proof}

In the rest of this paper, we will prove Theorem \ref{mainpropuni}.

\section{Some oscillatory integrals}

In this section, we collect the results on the decay rate of some oscillatory integrals. It is regarded as generalization of the Van der Corput lemma in higher dimensions.  Oscillatory integrals of the following forms are studied in \cite{V}:
\begin{align*}
\int_{\re^2}\chi(\y)e^{i\l f(\y)}d\y\quad \text{as}\quad \l\to \infty.
\end{align*}
For our purpose, we need the decay rates for the Fourier transform of the surface measure. To do this, we use the recent result by Ikromov and  M\"uller \cite{IM}. To prove the decay of such integrals, we need the following elementary lemma.

\begin{lem}\label{elzero}
Let $\a,\b\in \re\setminus \{0\}$ and define $f_1,f_2,f_3:\mathbb{S}=\{\y\in \re^2\mid |\y|=1\}\to \mathbb{C}$ by
\begin{align*}
f_1(\y)=\a \y_1^2\y_2+\b \y_1\y_2^2,\,\, f_2(\y)=\a \y_1^3+\b \y_2^2,\,\, f_3(\y)=\a \y_1^2\y_2+\b\y_2^2.
\end{align*}
Then any zeros of the functions $f_1,f_2,f_3$ are simple.
\end{lem}

\begin{proof}
We denote $\y_1=\cos\theta$ and $\y_2=\sin\theta$. Let $\f\in [0,2\pi)\setminus \{0,\frac{\pi}{2},\pi, \frac{3\pi}{2}\}$ be satisfying $\cos\f=\frac{\b}{\a^2+\b^2}$ and $\sin\f=\frac{\a}{\a^2+\b^2}$.
Then we write
\begin{align*}
f_1=&\cos\theta\sin\theta (\a\cos\theta+\b\sin\theta)=\sqrt{\a^2+\b^2}\cos\theta\sin\theta \sin (\theta+\f),\\
f_2=&\a\cos^3\theta+\b(1-\cos^2\theta),  \\
f_3=&-\a\sin^3\theta+\b\sin^2\theta+\a.
\end{align*}
Since the zeros of $\cos \theta$, $\sin \theta$ and $\sin (\theta+\f)$ are simple and since these zeros are distinct, it follows that the zeros of $f_1$ are simple. A simple calculation gives
\begin{align*}
\frac{df_2}{d\theta}=-\sin\theta \cos\theta(3\a \cos\theta-2\b).
\end{align*}
Thus we have
\begin{align*}
f_2(\theta)=\frac{df_2}{d\theta}(\theta)=0\Rightarrow \cos\theta=\pm\sqrt{3}.
\end{align*}
Since $|\cos\theta|\leq 1$, then the zeros of $f_2$ are simple. Finally, we have
\begin{align*}
f_3=-\a \sin\theta(\sin\theta-\frac{\b}{2\a}-\sqrt{\frac{\b^2}{4\a^2}+1})(\sin\theta-\frac{\b}{2\a}+\sqrt{\frac{\b^2}{4\a^2}+1})
\end{align*}
which only has simple zeros.
\end{proof}

The next proposition is a consequence of \cite[Theorem1.1]{IM}.

\begin{prop}\label{osidecay}
Let $f$ be a real-valued smooth function near $0\in \re^2$. For $\chi\in C_c^{\infty}(\re^d)$ supported near $0$, define
\begin{align*}
I(x)=\int_{\re^2}\chi(\y)e^{2\pi i(x_1\y_1+x_2\y_2+x_3f(\y))}d\y,\quad x\in \re^3.
\end{align*}
If the support of $\chi$ is close to $0$, the following holds.

\noindent $(i)$ Suppose that $f$ can be written as
\begin{align*}
f(\y)=f(0)+\sum_{j=1}^2\pa_{\y_j}f(0)\y_j+\a_{12}\y_1^2\y_2+\a_{21}\y_1\y_2^2+O(|\y|^4)\quad  \text{as}\quad |\y|\to 0
\end{align*}
with $\a_{12},\a_{21}\in \re\setminus \{0\}$. Then we have $I(x)=O(|x|^{-\frac{2}{3}})$ as $|x|\to \infty$.

\noindent$(ii)$ Suppose that $f$ can be written as
\begin{align*}
f(\y)=&f(0)+\sum_{j=1}^2\pa_{\y_j}f(0)\y_j+\a_{2}\y_2^2+\sum_{\substack{i+j+k=3\\ i\leq j\leq k}}\a_{ijk}\y_i\y_j\y_k\\
&+\sum_{\substack{i+j+k+m=4\\ i\leq j\leq k\leq m}}\a_{ijkm}\y_i\y_j\y_k\y_m+O(|\y|^5)\quad  \text{as}\quad |\y|\to 0
\end{align*}
with $\a_{2}\in \re\setminus \{0\}$ and $\a_{ijk}, \a_{ijkm}\in\re$. We assume
\begin{align*}
\a_{111}= 0\Rightarrow \a_{112}\neq 0\,\, \text{and}\,\, \a_{1111}= 0.
\end{align*}
Then we have $I(x)=O(|x|^{-\frac{5}{6}})$ as $|x|\to \infty$ if $\a_{111}\neq 0$ and $I(x)=O(|x|^{-\frac{3}{4}})$ as $|x|\to \infty$ otherwise.
\end{prop}

\begin{rem}
The results in \cite{V} imply that the above estimates are sharp for $x_1=x_2=0$ and $|x_3|\to \infty$ at least if $f$ is analytic.
\end{rem}

\begin{proof}
We may assume $f(0)=0$. Moreover, changing of the variable $x'_j=x_j+\pa_{\y_j}f(0)x_3$ ($j=1,2$) and $x_3'=x_3$, we may also assume $\pa_{\y_j}f(0)=0$ for $j=1,2$.

$(i)$ We use some notations and definitions from \cite[before Theorem 1.1]{IM}. Let $f_{\mathrm{pr}}$ be the principal part of $f$, $\pi(f)$ be the principal face, $d(f)$ be the Newton distance, $h(f)$ be the height of $f$ and $\n(f)$ be Varchenko's exponent of $f$. By a simple calculation, we have
\begin{align*}
\pi(f)=\{\y\in \re^2\mid \y_1\geq 1,\,\,\y_2\geq 2,\,\, \y_2= -\y_1+3\},\,\, f_{\mathrm{pr}}(\y)=\a_{12}\y_1^2\y_2+\a_{21}\y_1\y_2^2,\,\, d(f)=\frac{3}{2}.
\end{align*}
Moreover, it turns out that the coordinate $\y$ is adapted in the sense of \cite{IM}. In fact, it follows that $\pi(f)$ is the compact edge and $m(f_{\mathrm{pr}})=1<\frac{3}{2}=d(f)$ (this follows from Lemma \ref{elzero}), where $m(f_{\mathrm{pr}})$ is the vanishing order of $f_{\mathrm{pr}}|_{\mathbb{S}^1}$. This implies that $f$ satisfies the condition $(a)$ in \cite[before Lemma 1.5]{IM} and hence the coordinate $\y$ is adapted. This implies $h(f)=d(f)=\frac{3}{2}$. Since $h(f)<2$, we have $\n(f)=0$ by its definition. Now our claim follows from \cite[Theorem 1.1]{IM}.

$(ii)$ First, we assume $\a_{111}\neq 0$. By Lemma \ref{elzero}, we have
\begin{align*}
&\pi(f)=\{\y\in \re^2\mid \y_1\geq 0,\,\,\y_2\geq 0,\,\, \y_2= -\frac{2}{3}\y_1+2\},\quad f_{\mathrm{pr}}(\y)=\a_{2}\y_2^2+\a_{111}\y_1^3,\\
&d(f)=h(f)=\frac{6}{5},\quad m(f_{\mathrm{pr}})=1.
\end{align*}
Since $h(f)<2$, we obtain $\n(f)=0$ and $I(x)=O(|x|^{-\frac{5}{6}})$. Next, we assume $\a_{111}=0$. Since $\a_{1111}=0$, we have
\begin{align*}
&\pi(f)=\{\y\in \re^2\mid \y_1\geq 0,\,\,\y_2\geq 1,\,\, \y_2= -\frac{1}{2}\y_1+2\},\quad f_{\mathrm{pr}}(\y)=\a_{2}\y_2^2+2\a_{112}\y_1^2\y_2,\\
&d(f)=h(f)=\frac{4}{3},\quad m(f_{\mathrm{pr}})=1.
\end{align*}
Since $h(f)<2$, we obtain $\n(f)=0$ and $I(x)=O(|x|^{-\frac{3}{4}})$.
\end{proof}

\begin{rem}
When $\a_{111}=0$ in $(ii)$, the condition $\a_{1111}=0$ is necessary since in general, the principal part is written as $f_{\mathrm{pr}}(\y)=\a_2\y_2^2+2\a_{112}\y_1^2\y_2+\a_{1111}\y_1^4$.  As is pointed out by J.C. Cuenin and I.A. Ikromov, the optimal decay of $I$ is $O(|x|^{-\frac{1}{2}})$ for  the phase function $f(\y)=(\y_2-\y_1^2)^2=\y_2^2-2\y_1^2\y_2+\y_2^4$. 
\end{rem}

\section{Geometry of hypersurfaces}

In this section, we study the geometry of the Fermi surface $M_{\l}$ for $\l\in [0,4d]$. 
\subsection{General theory}
Let $M\subset \T^3$ or $M\subset \re^3$ be an embedded hypersurface of codimension $1$. Let $q\in M$. We may assume that there exist an open neighborhood $U\subset \T^3$ or $U\subset \re^3$ of $q$, an open set $V\subset \re^{2}$ and a smooth function $f:V\to \re$ such that $M\cap U=\{(\x', f(\x'))\mid \x'\in V\}$. 
We compute the induced Riemannian metric $g$ on $M$, the unit normal $\n$, the second fundamental form $A(\x')=(A(\x'))_{i,j=1}^{2}$ and the Gaussian curvature $K_1(\x')$:
\begin{align}
g&=\sum_{j=1}^{2}(1+\pa_{\x_j}f(\x')^2)d\x_j^2+2\pa_{\x_1}f(\x')\pa_{\x_2}f(\x')d\x_1d\x_2,\nonumber\\
\n(\x')&=\frac{1}{\sqrt{1+|\nabla_{\x'}f(\x')|^2}}
\begin{pmatrix}
-\nabla_{\x'}f(\x')\\
1
\end{pmatrix},\,\,A_{ij}(\x')=\frac{\pa_{\x_i}\pa_{\x_j}f(\x')}{\sqrt{1+|\nabla_{\x'}f(\x')|^2}},\label{nuab}\\
K_1(\x')&=\frac{\det (A_{ij}(\x'))}{\det g(\x')}=\frac{\det\pa_{\x_i}\pa_{\x_j}f(\x')}{(1+|\nabla_{\x'}f(\x')|^2)^2}.\label{Kab}
\end{align}

\begin{lem}\label{nonvaneq}
We denote the Gaussian curvature  at $\x\in M\cap U\subset \T^3$ by $K(\x)$, that is $K_1(\x')=K(\x',f(\x'))$ for $\x=(\x', f(\x'))\in M\cap U$. Then it follows that $\nabla_{\x'} K_1(\x')\neq 0$ if and only if 
\begin{align*}
(\nabla_{\x} h_0\times \nabla_{\x}K)(\x', f(\x'))\neq 0.
\end{align*}
\end{lem}

\begin{proof}
We recall $\nabla h_0$ is the unit normal of $M_{\l}$ and is parallel to the vector 
\begin{align*}
(-\pa_{\x_1}f, -\pa_{\x_2}f, 1).
\end{align*}
We learn
\begin{align*}
\nabla_{\x'} K_1(\x')=(\nabla_{\x'}K)(\x', f(\x'))+ (\pa_{\x_3}K)(\x', f(\x'))\nabla_{\x'}f(\x')
\end{align*}
and
\begin{align*}
\begin{pmatrix}
-\pa_{\x_1}f\\
-\pa_{\x_2}f\\
1
\end{pmatrix}\times \nabla_{\x} K(\x', f(\x'))=
\begin{pmatrix}
-\pa_{\x_3}K \pa_{\x_2}f-\pa_{\x_2}K\\
\pa_{\x_3}K\pa_{\x_1}f+\pa_{\x_1}K\\
-\pa_{\x_1}f\pa_{\x_2}K+\pa_{\x_2}f\pa_{\x_1}K
\end{pmatrix}.
\end{align*}
It follows from this calculation that $(\nabla_{\x} h_0\times \nabla_{\x}K)(\x', f(\x'))= 0$ implies $\nabla_{\x'} K_1(\x')= 0$.
A simple calculation implies that $\nabla_{\x'}K_1(\x')=0$ gives $-\pa_{\x_1}f\pa_{\x_2}K+\pa_{\x_2}f\pa_{\x_1}K=0$ at $\x=(\x',f(\x'))$. This completes the proof.
\end{proof}

It is useful to calculate the Taylor expansion of $f$ in terms of information about the Hessian of $f$:
\begin{lem}\label{matrixTaylor}
Let $V\subset \re^2$ be an open set and $f\in C^{\infty}(V;\re)$. Moreover, the $2\times 2$-matrix $B(\x')$ is defined by $B(\x')=(\pa_{\x_i}\pa_{\x_j}f(\x'))_{j,k=1}^2$. Suppose that there exist smooth functions $\l_{\pm}(\x')\in C^{\infty}(V; \re)$ and a orthogonal matrix 
\begin{align*}
U(\x')=\begin{pmatrix}
u_+(\x')&u_-(\x')
\end{pmatrix},\,\, u_{\pm}(\x')\in C^{\infty}(V; \re^2)\,\, \text{with}\,\, u_a(\x')\cdot u_{b}(\x')=\d_{ab}\,\, a,b\in \{\pm\}
\end{align*}
such that
\begin{align*}
U(\x')^{-1}B(\x')U(\x')=\begin{pmatrix}
\l_+(\x')&0\\
0&\l_-(\x')
\end{pmatrix}.
\end{align*}
Let $p\in V$. Set $U=U(p)$ and introduce the variable
\begin{align*}
\y=U^{-1}(\x'-p).
\end{align*}
Then we have
\begin{align*}
f(\x')=&f(p)+\pa_{\x'}f(p)\cdot U\y+\frac{1}{2}(\l_+(p)\y_1^2+\l_-(p)\y_2^2)\\
&+\frac{1}{3!}(u_+(p)\cdot(\nabla_{\x'}\l_+)(p)\y_1^3+u_-(p)\cdot(\nabla_{\x'}\l_-)(p)\y_2^3  )\\
&+\frac{1}{3!}(3u_-(p)\cdot(\nabla_{\x'}\l_+)(p)\y_1^2\y_2+3u_+(p)\cdot(\nabla_{\x'}\l_-)(p)\y_1\y_2^2  )\\
&+O(|\y|^4)
\end{align*}
as $\y\to 0$.

\end{lem}

\begin{proof}
We note
\begin{align*}
(\x'-p)\cdot \pa_{\x'}^2f(p)(\x'-p)=\y \cdot {}^tU B(p) U\y=\l_+(p)^2\y_1^2+\l_-(p)^2\y_2^2.
\end{align*}
Thus, it suffices to prove
\begin{align*}
\pa_{\y_1}^3f(p)= u_+(p)\cdot(\nabla_{\x'}\l_+)(p),\,\, \pa_{\y_1}^2\pa_{\y_2}f(p)=u_-(p)\cdot(\nabla_{\x'}\l_+)(p),\\
\end{align*}
\vspace{-1.1cm}
\begin{align*}
\pa_{\y_1}\pa_{\y_2}^2f(p)=u_+(p)\cdot(\nabla_{\x'}\l_-)(p),\,\,\pa_{\y_2}^3f(p)=u_-(p)\cdot(\nabla_{\x'}\l_-)(p).
\end{align*}
To see this, we observe
\begin{align*}
U(p)^{-1}\pa_{\x'}^2f(\x')U(p)=\pa_{\y}^2f(\x').
\end{align*}
This implies
\begin{align*}
\begin{pmatrix}
\l_{+}(\x')  &0\\
0&\l_-(\x')
\end{pmatrix}=U(\x')^{-1}U(p)\pa_{\y}^2f(\x')U(p)^{-1}U(\x')
\end{align*}
Differentiating in $\x'$ and substituting $\x'=p$, we have
\begin{align}
\begin{pmatrix}
\pa_{\x'}\l_{+}(p)  &0\\
0&\pa_{\x'}\l_-(p)
\end{pmatrix}=&\pa_{\x'}\pa_{\y}^2f(\x')+U(p)^{-1}U(p)\pa_{\y}^2f(p)U(p)^{-1}(\pa_{\x'}U)(p)\nonumber\\
&-U(p)^{-1}(\pa_{\x'}U)(p)U(p)^{-1}U(p)  \pa_{\y}^2f(p)U(p)^{-1}U(p)\nonumber\\
=&\pa_{\x'}\pa_{\y}^2f(\x')+\pa_{\y}^2f(p)U(p)^{-1}(\pa_{\x'}U)(p)\nonumber\\
&-U(p)^{-1}(\pa_{\x'}U)(p) \pa_{\y}^2f(p).\label{matrixcal1}
\end{align}
Using $(\pa_{\x'}|u_{\pm}(\x')|^2)|_{\x'=p}=\pa_{\x'}1=0$ and $(\pa_{\x'}u_{+}(\x')\cdot u_-(\x'))|_{\x'=p}=\pa_{\x'}0=0$, we have
\begin{align*}
u_+(p)\cdot \pa_{\x'} u_+(p)=u_-(p)\cdot \pa_{\x'} u_-(p)=0,\quad u_+(p)\cdot \pa_{\x'} u_-(p)+\pa_{\x'} u_+(p)\cdot u_-(p)=0.
\end{align*}
This implies
\begin{align*}
U(p)^{-1}(\pa_{\x'}U)(p)=&\begin{pmatrix}
u_+(p)\cdot \pa_{\x'} u_+(p)&u_+(p)\cdot \pa_{\x'} u_-(p)\\
u_-(p)\cdot \pa_{\x'} u_+(p)& u_-(p)\cdot \pa_{\x'} u_-(p)
\end{pmatrix}\\
=&
\begin{pmatrix}
0&u_+(p)\cdot \pa_{\x'} u_-(p)\\
u_-(p)\cdot \pa_{\x'} u_+(p)&0
\end{pmatrix}.
\end{align*}
Setting $a=u_+(p)\cdot \pa_{\x'} u_-(p)=-u_-(p)\cdot \pa_{\x'} u_+(p)$, $A=\pa_{\y_1}^2f(p)$ and $B=\pa_{\y_2}^2f(p)$, we have
\begin{align}
\pa_{\y}^2&f(p)U(p)^{-1}(\pa_{\x'}U)(p)-U(p)^{-1}(\pa_{\x'}U)(p) \pa_{\y}^2f(p)\nonumber\\
=&\begin{pmatrix}
A&0\\
0&B
\end{pmatrix}\begin{pmatrix}
0&a\\
-a&0
\end{pmatrix}-\begin{pmatrix}
0&a\\
-a&0
\end{pmatrix}\begin{pmatrix}
A&0\\
0&B
\end{pmatrix}=\begin{pmatrix}
0&a(A-B)\\
a(A-B)&0
\end{pmatrix}.\label{matrixcal2}
\end{align}
It follows from $(\ref{matrixcal1})$ and $(\ref{matrixcal2})$ that
\begin{align*}
\pa_{\x'}\l_{+}(p) =(\pa_{\x'}\pa_{\y_1}^2)f(p),\,\, \pa_{\x'}\l_{-}(p) =(\pa_{\x'}\pa_{\y_2}^2)f(p)
\end{align*}
Using $\pa_{\y_1}=u_+(p)\cdot \pa_{\x'}$ and $\pa_{\y_2}=u_-(p)\cdot \pa_{\x'}$, we complete the proof.

\end{proof}

\subsection{Geometry of the Fermi surface}

In the following, we consider the Fermi surface $M=M_{\l}=h_0^{-1}(\{\l\})$.
We fix some notations.
For $j=1,2,3$, we set
\begin{align*}
a_j=a_j(\x)=\cos2\pi \x_j,\,\, b_j=b_j(\x)=\sin 2\pi \x_j.%\,\, c_j=c_j(\x)=\tan 2\pi \x_j. 
\end{align*}
Set
\begin{align*}
E_{\l}=3-\l/2\in (-3,3)\quad \text{for}\quad 0<\l<12.
\end{align*}
From the expression $(\ref{multh_0})$, we have
\begin{align}\label{Fercal}
M_{\l}=\{\x\in \T^3\mid a_1+a_2+a_3=E_{\l}\},\,\mathrm{Cr}(h_0)=\{\x\in \T^3\mid b_j=0,\,\, j=1,2,3\},
\end{align}
where we recall that $\mathrm{Cr}(h_0)$ is defined in $(\ref{crih_0})$.
We define 
\begin{align*}
&K(\x)\in C^{\infty}(\T^3\setminus \mathrm{Cr}(h_0);\re):\, \text{the Gaussian curvature of $M_{\l}$ at $\x$},\\
&\n(\x)\in C^{\infty}(\T^3\setminus \mathrm{Cr}(h_0); \re^3):\, \text{the unit normal of $M_{\l}$ at $\x$}, 
\end{align*}
where the smoothness of $K$ follows from the implicit function theorem. For $\x\in M_{\l}\cap \{\pa_{\x_3}h_0(\x)\neq 0\}=\{b_3\neq 0\}$,
we write
\begin{align*}
\x=(\x',f_{\l}(\x')),\quad K(\x',f_{\l}(\x'))=K_1(\x').
\end{align*}
We note that the map $(\x',\l)\mapsto f_{\l}(\x')$ is smooth by virtue of the implicit function theorem.

We can calculate $K$ and $\n$ explicitly:
\begin{lem}
We have
\begin{align}
K(\x)=&\frac{4\pi^2(a_1a_2b_3^2+a_2a_3b_1^2+a_3a_1b_2^2)}{(b_1^2+b_2^2+b_3^2)^2},\,\, \n(\x)=\frac{1}{\sqrt{b_1^2+b_2^2+b_3^2}}(b_1,b_2,b_3)  \label{Kcal}.
\end{align}
\end{lem}

\begin{proof}
Let $\x_0\in \T^3\setminus \mathrm{Cr}(h_0)$ and $U$ be a small neighborhood of $\x_0$. We prove $(\ref{Kcal})$ at $\x_0$. Set $\l=h_0(\x_0)$.  By permutating the coordinate, we may assume $b_3(\x)\neq 0$. By the implicit function theorem, $U\cap M_{\l}$ has a graph representation: 
\begin{align*}
U\cap M_{\l}=\{(\x', f_{\l}(\x'))\}.
\end{align*}
Differentiating $h_0(\x',f_{\l}(\x'))=\l$ twice, for $j=1,2$, we have
\begin{align}\label{fdif1}
\pa_{\x_j}f_{\l}(\x')=-\frac{b_j}{b_3},   \pa_{\x_j}^2f_{\l}(\x')=-\frac{2\pi}{b_3^3}(a_jb_3^2+a_3b_j^2),\pa_{\x_1}\pa_{\x_2}f_{\l}(\x')=-\frac{2\pi b_1b_2a_3}{b_3^3}.
\end{align}
This implies 
\begin{align*}
(1+|\pa_{\x'}f(\x')|^2)^2=&\frac{(b_1^2+b_2^2+b_3^2)^2}{b_3^4}, \\
\det \pa_{\x'}^2f_{\l}(\x'):=&\pa_{\x_1}^2f(\x')\pa_{\x_2}^2f(\x')-(\pa_{\x_1}\pa_{\x_2}f(\x'))^2\\
=&\frac{4\pi^2}{b_3^4}(a_1a_2b_3^2+a_2a_3b_1^2+a_3a_1b_2^2).
\end{align*}
Substituting these relations into $(\ref{nuab})$ and $(\ref{Kab})$, we obtain $(\ref{Kcal})$.
\end{proof}

Now we determine all points where the Gaussian curvature vanishes.

\begin{prop}\label{Gauvanish}
Let $0<\l<12$. 

\noindent$(i)$
We have
\begin{align}
K^{-1}(0)\cap M_{\l}=&\{a_1a_2+a_2a_3+a_3a_1=a_1a_2a_3(a_1+a_2+a_3)\}\label{vani} \\
=&\{a_1=a_2=0\}\cup \{a_2=a_3=0\}\cup\{a_3=a_1=0\}\nonumber \\
&\cup\{a_1+a_2+a_3=1/a_1+1/a_2+1/a_3,\,\, a_1, a_2, a_3\neq 0\}.\nonumber
\end{align}
Moreover, if $K(\x)= 0$ with $\x\in M_{\l}\setminus \mathrm{Cr}(h_0)$, then $4\leq \l\leq 8$ holds.

\noindent $(ii)$ All principal curvatures of $M_{\l}\setminus \mathrm{Cr}(h_0)$ at $\x$ vanish  if and only if $\x_j\in \{1/4,3/4\}$ for $j=1,2,3$ and $\l=6$.

\noindent$(iii)$ The Gaussian curvature $K(\x)$ on $M_{6}$ vanishes if and only if $\x_j\in \{1/4,3/4\}$ for $j=1,2,3$.
\end{prop}

\begin{proof}
$(i)$
The first part of $(i)$ immediately follows from the representation $(\ref{Kcal})$ and the relations $a_j^2+b_j^2=1$ for $j=1,2,3$.

Next, we prove that $K(\x)= 0$ with $\x\in M_{\l}$ implies $4\leq \l\leq 8$. Let $\x\in K^{-1}(0)\cap M_{\l}$ and set
\begin{align*}
f(t)=t^3-E_{\l}t^2+(E_{\l}a_1a_2a_3)t-a_1a_2a_3.
\end{align*}
Then $(\ref{Fercal})$ and $(\ref{vani})$ imply that $a_1$, $a_2$ and $a_3$ are all zeros of $f$. Since $\lim_{t\to \pm\infty}f(t)=\pm \infty$ and $a_j\in [-1,1]$, we have $f(1)\geq 0$ and $f(-1)\leq 0$, which implies
\begin{align*}
(1-E_{\l})(1-a_1a_2a_3)\geq 0,\,\, (1+E_{\l})(1+a_1a_2a_3)\leq 0.
\end{align*}
These inequalities with $|a_1a_2a_3|\leq 1$ gives $-1\leq E_{\l}\leq 1$, which is equivalent to $4\leq \l\leq 8$.

$(ii)$ Next, we prove the part $(ii)$. Let $\x\in M_{\l}\setminus \mathrm{Cr}(h_0)$ with $4\leq \l\leq 8$. By permutating the coordinate, we may assume that $b_3(\x)\neq 0$ and that we can write $U\cap M_{\l}=\{(\x',f_{\l}(\x'))\}$. Then all principal curvatures of $M_{\l}$ at $\x$ vanishes if and only if $\pa_{\x_k}\pa_{\x_l}f(\x')=0$ for each $k,l=1,2$. This is also equivalent to 
\begin{align}\label{vanish}
\d_{kl}(1-a_3^2)a_k+b_kb_l\sqrt{1-a_k^2}\sqrt{1-a_l^2}a_3=0,\quad k,l=1,2.
\end{align}
Since it is easy to see that $\x_j\in\{1/4,3/4\}$ for $j=1,2,3$ imply $(\ref{vanish})$, then we prove that $(\ref{vanish})$ implies $\x_j\in\{1/4,3/4\}$ for $j=1,2,3$. Recall that $|a_3|\neq 1$ since we assume $\pa_{\x_3}h_0\neq 0$ on $M_{\l}$. 

If we suppose $a_3=0$, then $(\ref{vanish})$ with $k=l=1,2$ imply that $a_k=0$ for $k=1,2$ and hence $\x_k\in \{1/4,3/4\}$ for $k=1,2,3$. 

If we suppose $a_3\neq 0$, then $(\ref{vanish})$ with $k=1$ and $l=2$ imply that either $|a_1|$ or $|a_2|$ is equal to $1$. Then it follows from $a_3\neq 0$ and from $(\ref{vanish})$ with $k=l=1$ or $k=l=2$ that $|a_3|=1$. Thus we obtain $|a_k|=1$ for $k=1,2$ by $(\ref{vanish})$ with $k=l=1,2$. However, this contradicts to $\nabla h_0(\x)\neq0$ and hence $a_3=0$.

$(iii)$ Finally, we prove the part $(iii)$. We note that $\l=6$ is equivalent to $E_{\l}=0$. $(\ref{Fercal})$ and $(\ref{vani})$ implies
\begin{align}\label{6eq}
a_1+a_2+a_3=a_1a_2+a_2a_3+a_3a_1=0
\end{align}
at $\x\in M_{6}$. Then, it follows that $a_1,a_2, a_3$ are the solutions to the equation
\begin{align}\label{aleq}
t^3-a_1a_2a_3=0.
\end{align}
If $a_1a_2a_3=0$, then we have $a_1=a_2=a_3=0$ and hence $\x_j\in \{1/4,3/4\}$ holds for $j=1,2,3$. We suppose $a_1a_2a_3\neq 0$ and deduce a contradiction. Substituting $(\ref{aleq})$ into $t=a_1$, $a_2$ and $a_3$, we have $a_1^3=a_2^3=a_3^3=a_1a_2a_3$. This gives $a_1^2=a_2a_3$, $a_2^2=a_3a_1$ and $a_3^2=a_1a_2$. Combining these relations with $(\ref{aleq})$, we obtain $a_1=a_2=a_3$. Thus $(\ref{6eq})$ implies $a_1=a_2=a_3=0$. This is a contradiction.

\end{proof}

\begin{lem}\label{nonvanih0}
Let $4<\l<8$ with $\l\neq 6$ and $\x_*\in M_{\l}\cap K^{-1}(0)$. Then we obtain
\begin{align}\label{vecneq}
\nabla_{\x} h_0(\x_*)\times \nabla_{\x}K(\x_*)\neq 0.
\end{align}
In particular, from Lemma \ref{nonvaneq}, we have $(\nabla_{\x'}K_1)(\x_{*}')\neq 0$, where $\x_*=(\x_*',f_{\l}(\x'))$.
\end{lem}

\begin{proof}
We set
\begin{align*}
\n(\x)=(b_1,b_2,b_3),\,\, \tilde{K}(\x)=a_1a_2b_3^2+a_2a_3b_1^2+a_3a_1b_2^2.
\end{align*}
Using $\tilde{K}(\x_*)=0$, $\nabla h_0\parallel \n$ and $(\ref{Kcal})$, we see that $(\ref{vecneq})$ is equivalent to
\begin{align*}
\n(\x_*)\times \nabla_{\x} \tilde{K}(\x_*)\neq 0.
\end{align*}
A direct computation gives
\begin{align*}
\n(\x)\times \nabla_{\x} \tilde{K}(\x)=-2\pi
\begin{pmatrix}
b_2b_3(a_2-a_3)(1-a_1(a_1+a_2+a_3))\\
b_3b_1(a_3-a_1)(1-a_2(a_1+a_2+a_3))\\
b_1b_2(a_1-a_2)(1-a_3(a_1+a_2+a_3))
\end{pmatrix}.
\end{align*}
We note that $(a_1+a_2+a_3)(\x_*)=E_{\l}$. Moreover, it follows that $\l\in (4,8)$ is equivalent to $E_{\l}\in (-1,1)$. These relations with $-1\leq a_j\leq 1$ imply that for $\x\in M_{\l}$, $\n(\x)\times \nabla_{\x} \tilde{K}(\x)=0$ is equivalent to
\begin{align}\label{vanieq}
b_2b_3(a_2-a_3)=b_3b_1(a_3-a_1)=b_1b_2(a_1-a_2)=0\,\, \text{at}\,\, \x.
\end{align}
Since $\x_*\in M_{\l}\cap K^{-1}(0)$ with $\l\in (4,8)\setminus\{6\}$, $(\ref{Kcal})$ implies that $(\ref{vanieq})$ does not hold at $\x_*$. This completes the proof.

\end{proof}

\subsection{Concrete description of the Fourier transform of the surface measure}

Now we set 
\begin{align}\label{Bdef}
\pa_{\x'}^2f(\x')=-2\pi\begin{pmatrix}
\frac{a_1b_3^2+a_3b_1^2}{b_3^3}&  \frac{b_1b_2a_3}{b_3^3}\\
\frac{b_1b_2a_3}{b_3^3}& \frac{a_2b_3^2+a_3b_2^2}{b_3^3}
\end{pmatrix}
%=:-2\pi\begin{pmatrix}
%\a&  \b\\
%\b& \c
%\end{pmatrix}
=:B(\x').
\end{align}

\begin{prop}\label{graphrep}
Let $\x_*\in M_{\l}\cap K^{-1}(0)$ with $\l\in (4,8)$ with $b_3(\x_*)\neq 0$ and $U\subset \T^3$ be a small neighborhood of $\x$ such that $U\cap M_{\l}$ has a graph representation: $U\cap M_{\l}=\{(\x',f_{\l}(\x'))\}$. Then we have

\noindent$(i)$ If $\l= 6$, then $f_{\l}$ has the following Taylor expansion near $\x_*=(\x'_*, f_{\l}(\x_*'))$:
\begin{align}\label{6Tay}
f_{\l}(\x')=f_{\l}(\x'_*)+(\pa_{\x'}f_{\l})(\x'_*)\cdot\y+\a_{12}\y_1^2\y_2+\a_{21}\y_1\y_2^2+R(\y),
\end{align}
where $\y=(\y_1,\y_2)=\x'-\x'_*$ and a real-valued function $R$ satisfies $|\pa_{\y}^{\c}R(\y)|\leq C|\y|^{\max(4-|\c|, 0)}$. Here $\a_{12}, \a_{21}\in \re\setminus \{0\}$.

\noindent $(ii)$ Suppose $\l\neq 6$. We regard $\x'-\x'_*$ as a vector in $\re^2$. Then there exists a $2\times 2$ unitary matrix $U$ such that
\begin{align*}
f_{\l}(\x')=&f_{\l}(\x'_*)+(\pa_{\x'}f_{\l})(\x'_*)\cdot U\y+\a_1\y_1^2+\a_2\y_2^2\\
&+(\a_{111}\y_1^3+3\a_{112}\y_1^2\y_2+3\a_{122}\y_1\y_2^2+\a_{222}\y_2^3)+\sum_{\substack{i+j+k+m=4 \\ i\leq j\leq k\leq m}}\a_{ijkm}\y_i\y_j\y_k\y_m+R(\y),
\end{align*}
where we set $\y=(\y_1,\y_2)=U^{-1}(\x'-\x'_*)$ and a real-valued function $R$ satisfies $|\pa_{\y}^{\c}R(\y)|\leq C|\y|^{\max(5-|\c|, 0)}$. Here $\a_1,\a_2,\a_{ij}\in \re$ for $i,j=1,2$ satisfies $(\a_1,\a_2)\neq (0,0)$ and
\begin{align*}
&\a_1\neq0 \Rightarrow (\a_{122},\a_{222})\neq (0,0),\\
&\a_2\neq0 \Rightarrow (\a_{111},\a_{112})\neq (0,0).
\end{align*}
Moreover, it follows that if $(\a_1,\a_{111})=(0,0)$ or $(\a_2,\a_{222})=(0,0)$ hold, then we have
\begin{align}
(a_1(\x_*),a_2(\x_*), a_3(\x_*))\in \{(0,0,E_{\l}), (0,E_{\l},0), (E_{\l},0,0)\}. \label{derhigh1}
\end{align}

\noindent $(iii)$ Suppose $\l\neq 6$. Let $\x_*$ satisfying 
\begin{align*}
a_1(\x_*)=a_3(\x_*)=0,\quad a_2(\x_*)=E_{\l}.
\end{align*}
Then we have
\begin{align*}
f_{\l}(\x')=&f_{\l}(\x'_*)+(\pa_{\x'}f_{\l})(\x'_*)\cdot \x'+\a_1\x_1^2+\a_2\x_2^2\\
&+(\a_{111}\x_1^3+3\a_{112}\x_1^2\x_2+3\a_{122}\x_1\x_2^2+\a_{222}\x_2^3)+\sum_{\substack{i+j+k+m=4 \\ i\leq j\leq k\leq m}}\a_{ijkm}\x_i\x_j\x_k\x_m+R(\x'),
\end{align*}
where
\begin{align*}
\a_1=\a_{111}=\a_{1111}=0,\quad \a_{112}\neq 0,
\end{align*}
hold and a real-valued function $R$ satisfies $|\pa_{\x'}^{\c}R(\x')|\leq C|\x'|^{\max(5-|\c|, 0)}$.
\end{prop}

\begin{proof}
$(i)$ Let $\x_*\in M_{6}\cap K^{-1}(0)$. Proposition \ref{Gauvanish} implies that $(\x_*)_j\in \{1/4, 1/3\}$ for each $j=1,2,3$ (which automatically implies $b_3(\x_*)\neq 0$). We prove $(\ref{6Tay})$ only for $(\x_*)_j=1/4$, $j=1,2,3$. The other cases are similarly proved. Differentiating $h_0(\x',f_{\l}(\x'))=6$ three times, we have $(\ref{fdif1})$ and
\begin{align}
\pa_{\x_m}\pa_{\x_k}\pa_{\x_l}f_{\l}(\x')=&\frac{4\pi^2}{b_3}((\d_{kl}\d_{lm}b_k-\frac{b_mb_lb_k}{b_3^2})\nonumber\\
&-\frac{a_3}{b_3^4}(\d_{kl}a_kb_mb_3^2+\d_{lm}a_lb_kb_3^2+\d_{mk}a_mb_lb_3^2+3b_mb_kb_la_3)  )\label{fthreedif}
\end{align}
for $j,k,l=1,2$. Substituting this into $\x_*=(1/4,1/4,1/4)$, we obtain
\begin{align*}
\pa_{\x_j}^3f_{\l}(\x_*')=0,\,\, \pa_{\x_1}^2\pa_{\x_2}f_{\l}(\x_*')=\pa_{\x_1}^2\pa_{\x_2}f_{\l}(\x_*')=-4\pi^2.
\end{align*}
for $j=1,2$. Taylor expanding $f_{\l}$, we obtain $(\ref{6Tay})$.

$(ii)$ We recall $B$ is the matrix defined in $(\ref{Bdef})$.  We denote the eigenvalues of $B$ at $\x'$ by $\l_+(\x')$ and $\l_-(\x')$. 
Now $(\ref{Kab})$, $(\ref{Bdef})$ and Lemma \ref{nonvanih0} imply $\nabla_{\x'}\det B(\x_*')\neq 0$. Thus we have
\begin{align*}
(\nabla_{\x'}\l_+)(\x_*')\l_-(\x_*')+(\nabla_{\x'}\l_-)(\x_*')\l_+(\x_*')\neq 0.
\end{align*}
Since $\l_+(\x_*')\l_{-}(\x_*')=0$ and $(\l_+(\x_*'), \l_-(\x_*'))\neq (0,0)$, we have
\begin{align}
&\l_+(\x_*')=0\Rightarrow \nabla_{\x'}\l_-(\x_*')\neq 0, \label{derlam+}\\
&\l_-(\x_*')=0 \Rightarrow \nabla_{\x'}\l_+(\x_*')\neq 0  \label{derlam-}.
\end{align}
Note that $\l_+(\x'_*)$ and $\l_-(\x'_*)$ are distinct by virtue of Proposition \ref{Gauvanish}. Then \cite[Theorem XII.4]{RS} implies that $\l_+(\x')$ and $\l_-(\x')$ are analytic near $\x_*'$ and the corresponding unit eigenvectors $u_+(\x')$ and $u_-(\x')$ can be chosen to be analytic near $\x_*'$. Now our claim follows from Lemma \ref{matrixTaylor}, $(\ref{derlam+})$ and $(\ref{derlam-})$. $(\ref{derhigh1})$ will be proved in Appendix \ref{appB}.

$(iii)$ It suffices to prove $\a_{112}\neq 0$ and $\a_1=\a_{111}=\a_{1111}=0$, that is,
\begin{align*}
\pa_{\x_1}^2\pa_{\x_2}f_{\l}(\x_*')\neq 0,\quad \pa_{\x_1}^2f_{\l}(\x_*')=\pa_{\x_1}^3f_{\l}(\x_*')=\pa_{\x_1}^4f_{\l}(\x_*')=0.
\end{align*}
The relations $\pa_{\x_1}^2\pa_{\x_2}f_{\l}(\x_*')\neq 0$ and $\pa_{\x_1}^2f_{\l}(\x_*')=\pa_{\x_1}^3f_{\l}(\x_*')=0$ directly follow from $(\ref{fdif1})$ and $(\ref{fthreedif})$.
Differentiating $h_0(\x',f_{\l}(\x'))=\l$ four times in $\x_1$-variable, we have 
\begin{align*}
-&8\pi^3a_1+(\pa_{\x_1}^4f_{\l})b_3+8\pi(\pa_{\x_1}^3f_{\l})(\pa_{\x_1}f_{\l})a_3+6\pi(\pa_{\x_1}^2f_{\l})^2a_3\\
&-24\pi^2(\pa_{\x_1}^2f_{\l})(\pa_{\x_1}f_{\l})^2b_3-12\pi^3(\pa_{\x_1}f_{\l})^4a_3=0.
\end{align*}
Substituting $a_1(\x_*)=a_3(\x_*)=\pa_{\x_1}^2f_{\l}(\x_*')=\pa_{\x_1}^3f_{\l}(\x_*')=0$, we obtain $\pa_{\x_1}^4f_{\l}(\x_*')=0$.

\end{proof}

\section{Proof of Theorem \ref{mainpropuni}}
By permutating the coordinate, we may assume $\pa_{\x_3}h_0(\x)\neq 0$ on $\supp \chi$. We use the following representation: 
\begin{align*}
\widehat{\chi d\s_{M_{\l}}}(\x)=\int_{\mathbb{T}^3}\chi(\x',f_{\l}(\x'))e^{2\pi i(x_1\x_1+x_2\x_2+x_3f_{\l}(\x'))}   \frac{d\x'}{|\nabla h_0(\x', f_{\l}(\x'))|},\quad \x'=(\x_1,\x_2),
\end{align*}
where we write $M_{\l}=\{(\x', f_{\l}(\x'))\}$ locally.

Theorem \ref{mainpropuni} $(i)$ and $(iv)$ directly follows from Proposition \ref{Gauvanish} $(i)$, $(ii)$ and the stationary phase theorem. See \cite[Chapter VIII, \S 3, Theorem 1]{S}. Theorem \ref{mainpropuni} $(ii)$ follows from Proposition \ref{osidecay} $(i)$ and Proposition \ref{graphrep} $(i)$. Moreover,  Proposition \ref{osidecay} $(ii)$ and Proposition \ref{graphrep} $(ii)$, $(iii)$ imply Theorem \ref{mainpropuni} $(ii)$ (if necessary, permuting the coordinate). We finish the proof.

%%%%%%%%%%%%%%%%%%%%%%%%%%%%%%%%%%%%%%%%%%%%%%%%%%%%%%%%%%%%%%%%%%%%%%%%%%%%%%%%%%%%%%%%%%%%%%%%%%%%%%%%%%%%%%%%%%%%%%%%%%%%%%%%%%%%%%%%%%%%%%%%%%%%%%%%%%%%%%%%%%%%%%%%%%%%%%%%%%%%%%%%%%%%%%%%%%%%%%%%%%%%%%%%%%%%%%%%%%%%%%%%%%%%%%%%%%%%%%%%%%%%%%%%%%%%%%%%%%%%%%%%%%%%%%%%%%%%%%%%%%%%%%%%%%%%%%%%%%%%%%%%%%%%%%%%%%%%%%%%%%%%%%%%%%%%%%%%%%%%%%%%%%%%%%%%%%%%%%%%%%%%%%%%%%%%%%%%%%%%%%%%%%%%%%%%%%%%%%%%%%%%%%%%%%%%%%%%%%%%%%%%%%%%%%%%%%%%%%%%%%%%%%%%%%

\begin{comment}

\section{Optimiality}

\begin{lem}
Suppose $E_{\l}=\pm 1$. If the Gaussian curvature of $M_{\l}$ vanishes at $\x\in M_{\l}$, then we have
\begin{align*}
\x\in \{a_1=\pm 1,\,\, a_2=-a_3\}\cup \{a_2=\pm 1,\,\, a_3=-a_1\}\cup \{a_3=\pm 1,\,\, a_1=-a_2\}.
\end{align*}

\end{lem}
\begin{proof}
If $E_{\m}=\pm 1$, then $a_1,a_2,a_3$ is the solutions to
\begin{align*}
t^3\mp t^2\pm a_1a_2a_3t-a_1a_2a_3=0
\end{align*}
by $(\ref{Fercal})$ and $(\ref{Kcal})$.
Since $t^3\mp t^2\pm a_1a_2a_3t-a_1a_2a_3=(t\mp 1)(t^2\pm a_1a_2a_3)=0$ and $a_1+a_2+a_3=\pm 1$, we obtain the desired result.
\end{proof}

\end{comment}

%%%%%%%%%%%%%%%%%%%%%%%%%%%%%%%%%%%%%%%%%%%%%%%%%%%%%%%%%%%%%%%%%%%%%%%%%%%%%%%%%%%%%%%%%%%%%%%%%%%%%%%%%%%%%%%%%%%%%%%%%%%%%%%%%%%%%%%%%%%%%%%%%%%%%%%%%%%%%%%%%%%%%%%%%%%%%%%%%%%%%%%%%%%%%%%%%%%%%%%%%%%%%%%%%%%%%%%%%%%%%%%%%%%%%%%%%%%%%%%%%%%%%%%%%%%%%%%%%%%%%%%%%%%%%%%%%%%%%%%%%%%%%%%%%%%%%%%%%%%%%%%%%%%%%%%%%%%%%%%%%%%%%%%%%%%%%%%%%%%%%%%%%%%%%%%%%%%%%%%%%%%%%%%%%%%%%%%%%%%%%%%%%%%%%%%%%%%%%%%%%%%%%%%%%%%%%%%%%%%%%%%%%%%%%%%%%%%%%%%%%%%%%%%%%%

\appendix

\section{Equivalence of uniform resolvent estimates}

Next elementary lemma follows from the H\"older inequality and the duality argument. 

\begin{lem}\label{unifequi}
Let $p\in [1,2]$ and $r\in [2,\infty]$ satisfying
\begin{align*}
\frac{1}{p}=\frac{1}{2}+\frac{1}{r}.
\end{align*}
Set $p'=p/(p-1)$. Then
\begin{align}\label{unifl^p}
\|A\|_{B(l^{p}(\ze^d),l^{p'}(\ze^d))}\leq C
\end{align}
is equivalent to
\begin{align}\label{unifpot}
\|W_1AW_2\|_{B(l^2(\ze^d))}\leq C\|W_1\|_{l^{r}(\ze^d)}\|W_2\|_{l^{r}(\ze^d)}\quad \text{for}\quad W_1, W_2\in l^r(\ze^d)
\end{align}
with a same constant $C>0$.
\end{lem}

\begin{proof}
Let $u\in l^2(\ze^d)$.
The H\"older inequality with $(\ref{unifl^p})$ implies
\begin{align*}
\|W_1AW_2u\|_{l^2(\ze^d)}\leq& \|W_1\|_{l^r(\ze^d)}\|AW_2u\|_{l^{p'}(\ze^d)}\\
\leq&C\|W_1\|_{l^r(\ze^d)}\|W_2u\|_{l^{p}(\ze^d)}\\
\leq&C\|W_1\|_{l^r(\ze^d)}\|W_2\|_{l^{r}(\ze^d)}\|u\|_{l^{2}(\ze^d)}.
\end{align*}
Thus we have $(\ref{unifpot})$. Conversely, assume $(\ref{unifpot})$ and fix $W_2\in l^r(\ze^d)$. First, we prove $\|AW_2\|_{B(l^2(\ze^d),l^{p'}(\ze^d))}\leq C\|W_2\|_{l^r(\ze^d)}$. Let $u,w$ be finitely supported functions. Then $(\ref{unifl^p})$ implies
\begin{align*}
|(w, AW_2u)_{l^2(\ze^d)}|=&|(|w|^{\frac{p}{2}}\sgn w, |w|^{1-\frac{p}{2}}AW_2u)_{l^2(\ze^d)}|\\
\leq&C\||w|^{\frac{p}{2}}\|_{l^2(\ze^d)}\||w|^{1-\frac{p}{2}}\|_{l^r(\ze^d)} \|W_2\|_{l^r(\ze^d)}\|u\|_{l^2(\ze^d)}\\
=&C\|W_2\|_{l^r(\ze^d)}\|w\|_{l^p(\ze^d)} \|u\|_{l^2(\ze^d)}.
\end{align*}
Thus we have $\|AW_2\|_{B(l^2(\ze^d),l^{p'}(\ze^d))}\leq C\|W_2\|_{l^r(\ze^d)}$. Similar argument also implies $(\ref{unifl^p})$.
\end{proof}

\section{Proof of $(\ref{derhigh1})$}\label{appB}

In this appendix, we prove $(\ref{derhigh1})$.
We introduce the notation 
\begin{align*}
a_j=a_j(\x)=\cos2\pi \x_j,\,\, b_j=b_j(\x)=\sin 2\pi \x_j,\,\, c_j=c_j(\x)=\tan 2\pi \x_j=\frac{b_j}{a_j} 
\end{align*}
and recall $M_{\l}=h_0^{-1}(\{\l\})$ and $E_{\l}=3-\l/2$. We note
\begin{align*}
E_{\l}\in (-1,1)\Leftrightarrow \l\in (4,8).
\end{align*}
For $\x\in M_{\l}$ with $b_3(\x)\neq 0$, we write
\begin{align*}
\x=(\x',f_{\l}(\x)).
\end{align*}
We recall 
\begin{align*}
B(\x')=\pa_{\x'}^2f_{\l}(\x')=-2\pi\begin{pmatrix}
\frac{a_1b_3^2+a_3b_1^2}{b_3^3}&  \frac{b_1b_2a_3}{b_3^3}\\
\frac{b_1b_2a_3}{b_3^3}& \frac{a_2b_3^2+a_3b_2^2}{b_3^3}
\end{pmatrix}
\end{align*}
and denote the eigenvalues of $B(\x')$ by $\l_+(\x')$ and $\l_-(\x')$ and the corresponding eigenvectors by $u_+(\x')$ and $u_-(\x')$. 

In order to prove $(\ref{derhigh1})$, by permuting the coordinate, it suffices to prove
\begin{align*}
\l_+(\x_*')=u_+(\x_*')\cdot \pa_{\x'}\l_+(\x_*')=0\Rightarrow (a_1(\x_*),a_2(\x_*), a_3(\x_*))\in \{(0,0,E_{\l}), (0,E_{\l},0), (E_{\l},0,0)\}
\end{align*}
if we suppose $E_{\l}\in (-1,1)\setminus \{0\}$. We recall $\l_-(\x_*')\neq 0$ if $\l_+(\x_*')=0$ by Proposition \ref{Gauvanish} and by the condition $E_{\l}\neq 0$. Since 
\begin{align*}
a_j=0\quad \text{for some}\,\, j=1,2,3\,\, \Rightarrow  (a_1,a_2, a_3)\in \{(0,0,E_{\l}), (0,E_{\l},0), (E_{\l},0,0)\},
\end{align*}
at $\x\in K^{-1}(\{0\})\cap M_{\l}$, we only need to prove
\begin{align}\label{suffice}
a_j\neq 0\quad \text{for all}\,\, j=1,2,3\,\, \text{and}\,\, \l_+(\x_*)=0 \Rightarrow u_+(\x_*')\cdot \pa_{\x'}\l_+(\x_*')\neq 0.
\end{align}
First, we compute the null eigenvector of $B(\x_*')$.
\begin{prop}\label{nullev}
Suppose $E_{\l}\in (-1,1)\setminus \{0\}$. If $\x_*=(\x_*', f_{\l}(\x_{*}'))\in M_{\l}\cap K^{-1}(\{0\})\setminus \left(\{a_1=0\}\cup \{a_2=0\}\cup \{b_3= 0\}\right)$ satisfies $\l_+(\x_*')=0$, then we have
\begin{align*}
u_+(\x_*')\parallel \begin{pmatrix}c_1(\x_*)\\ c_2(\x_*)\end{pmatrix}.
\end{align*}

\end{prop}

We will prove this proposition in the next subsection.
It follows from this proposition that under the condition $\l_+(\x_*')=0$, the equation $u_+(\x_*')\cdot \pa_{\x'}\l_+(\x_*')=0$ is equivalent to
\begin{align*}
\begin{pmatrix}c_1(\x_*)\\ c_2(\x_*)\end{pmatrix}\cdot (\pa_{\x}\det B)(\x_*')=0
\end{align*}
if $a_1(\x_*)\neq 0$ and $a_2(\x_*)\neq 0$ are satisfied. 
Since $\det B(\x_*')=0$, this equation is also equivalent to
\begin{align}\label{evdervan}
\begin{pmatrix}c_1(\x_*)\\ c_2(\x_*)\end{pmatrix}\cdot (\pa_{\x}\det(\frac{-b_3^4}{2\pi} B))(\x_*')=0.
\end{align}

\begin{lem}
Suppose $E_{\l}\in (-1,1)\setminus \{0\}$. For $\x=(\x', f_{\l}(\x'))\in M_{\l}$, we have
\begin{align*}
\pa_{\x}\det(\frac{-b_3^4}{2\pi} B)=-2\pi\begin{pmatrix}
b_1(a_3-a_1)(1-E_{\l}a_2)\\
b_2(a_3-a_2)(1-E_{\l}a_1)
\end{pmatrix}.
\end{align*}

\end{lem}

\begin{proof}
We set $M=\frac{-b_3^4}{2\pi} B$.
A direct calculation gives
\begin{align*}
M=&a_1a_2b_3^2+a_1a_3b_2^2+a_2a_3b_1^2.
\end{align*}
Since $\pa_{\x_j}a_3=-2\pi b_3(\pa_{\x_j}f_{\l})=2\pi b_j$ and $\pa_{\x_j}b_3=2\pi a_3(\pa_{\x_j}f_{\l})=-2\pi b_ja_3/b_3$ (recall $(\ref{fdif1})$), we obtain
\begin{align*}
\pa_{\x_1}M(\x)=&2\pi b_1(a_1-a_3)(1-a_2(a_1+a_2+a_3))\\
\pa_{\x_2}M(\x)=&2\pi b_2(a_2-a_3)(1-a_1(a_1+a_2+a_3)).
\end{align*}
This completes the proof.
\end{proof}

Now we compute the left hand side of $(\ref{evdervan})$. We assume $a_1\neq 0$, $a_2\neq 0$ and $a_3\neq 0$. Using the relations $a_jb_jc_j=b_j^2$, we have
\begin{align*}
\frac{1}{2\pi}\begin{pmatrix}c_1\\ c_2\end{pmatrix}\cdot (\pa_{\x}\det(\frac{-b_3^4}{2\pi} B))=&(b_1c_1+b_2c_2)a_3-b_1^2-b_2^2\\
&+E_{\l}\left(a_1a_2(b_1c_1+b_2c_2)-a_3(b_1c_1a_2+a_1b_2c_2) \right).
\end{align*}
Since $\sum_{j=1}^3b_jc_j=0$ which is proved in Lemma \ref{Gauperp} below, we obtain
\begin{align*}
\frac{1}{2\pi}\begin{pmatrix}c_1\\ c_2\end{pmatrix}\cdot (\pa_{\x}\det(\frac{-b_3^4}{2\pi} B))=&-a_3b_3c_3-b_1^2-b_2^2\\
&-E_{\l}(a_1a_2b_3c_3+a_3(b_1c_1a_2+a_1b_2c_2))\\
=&-b_1^2-b_2^2-b_3^2-E_{\l}\left(\frac{a_1a_2b_3^2}{a_3}+\frac{a_2a_3b_1^2}{a_1}+\frac{a_3a_1b_2^2}{a_2} \right)\\
=&-b_1^2-b_2^2-b_3^2-E_{\l}\left(\frac{a_1^2a_2^2+a_2^2a_3^2+a_3^2a_1^2}{a_1a_2a_3}-3a_1a_2a_3\right).
\end{align*}
From the relations (see $(\ref{vani})$)
\begin{align*}
a_1+a_2+a_3=E_{\l},\quad a_1a_2+a_2a_3+a_3a_1=E_{\l}a_1a_2a_3,
\end{align*}
we obtain
\begin{align*}
a_1^2a_2^2+a_2^2a_3^2+a_3^2a_1^2=E_{\l}a_1a_2a_3(E_{\l}a_1a_2a_3-2),\quad b_1^2+b_2^2+b_3^2=3-E_{\l}^2+2E_{\l}a_1a_2a_3.
\end{align*}
Thus we have
\begin{align*}
\frac{1}{2\pi}\begin{pmatrix}c_1\\ c_2\end{pmatrix}\cdot (\pa_{\x}\det(\frac{-b_3^4}{2\pi} B))=&E_{\l}^2-3-2E_{\l}a_1a_2a_3-E_{\l}\left(E_{\l}(E_{\l}a_1a_2a_3-2) -3a_1a_2a_3\right)\\
=&-a_1a_2a_3E_{\l}^3+3E_{\l}^2+a_1a_2a_3E_{\l}-3\\
=&(E_{\l}^2-1)(-a_1a_2a_3E_{\l}+3)\\
\neq&0
\end{align*}
since $E_{\l}\in (-1,1)$. This proves $(\ref{suffice})$.

\subsection{Proof of Proposition \ref{nullev}}

We need the following lemmas.

\begin{lem}\label{Gauperp}
Suppose $E_{\l}\in [-1,1]$. Then, for $\x\in M_{\l}\setminus\{a_1=0\}\cup \{a_2=0\}\cup \{a_3=0\}$, $K(\x)=0$ holds if and only if $b(\x)=(b_1(\x),b_2(\x),b_3(\x))\perp c(\x)=(c_1(\x),c_2(\x),c_3(\x))$,
where we recall $K(\x)$ is the Gaussian curvature at $\x\in M_{\l}$ which is defined in Lemma \ref{nonvaneq}.
\end{lem}

\begin{proof}
By virtue of $(\ref{Kcal})$, $K(\x)=0$ if and only if
\begin{align*}
a_1a_2b_3^2+a_2a_3b_1^2+a_3a_1b_2^2=0\quad \text{at}\quad \x.
\end{align*}
Since $b_j^2=a_jb_jc_j$, this equation is equivalent to
\begin{align*}
a_1a_2a_3(b_1c_1+b_2c_2+b_3c_3)=0\quad \text{at}\quad \x.
\end{align*}
This is also equivalent to $b(\x)\perp c(\x)$ under the conditions $a_1\neq 0$, $a_2\neq 0$ and $a_3\neq 0$.
\end{proof}

\begin{lem}\label{b_12nonvani}
Suppose $E_{\l}\in (-1,1)$. If $\x\in K^{-1}(\{0\})\cap M_{\l}$, then we have $(b_1(\x), b_2(\x))\neq (0,0)$. In particular, we obtain $(c_1(\x), c_2(\x))\neq (0,0)$.

\end{lem}

\begin{proof}
If $b_1(\x)=b_2(\x)=0$, then we have $|a_1(\x)|=|a_2(\x)|=1$. It follows that $a_1(\x)=a_2(\x)=\pm 1$ does not hold since these imply $a_3(\x)=E_{\l}\mp 2\notin [-1,1]$, which is a contradiction. Thus we have $(a_1(\x), a_2(\x))=(\pm 1, \mp 1)$ and $a_3(\x)=E_{\l}$ by $(\ref{Fercal})$. Using $(\ref{vani})$, we conclude $E_{\l}^2=1$, which is contradicts to $E_{\l}\in (-1,1)$.
\end{proof}

\begin{lem}\label{nullevlem}
Suppose $E_{\l}\in (-1,1)$. For $\x\in K^{-1}(\{0\})\cap M_{\l}\setminus\{a_1=0\}\cup \{a_2=0\}$ with $b_3(\x)\neq 0$, we have
\begin{align}\label{Beveq}
B(\x')\begin{pmatrix}
c_1\\
c_2
\end{pmatrix}=0,
\end{align}
where we write $\x=(\x', f_{\l}(\x'))\in M_{\l}$. In particular, the vector $(c_1,c_2)$ is the eigenvector of the matrix $B$ with $0$-eigenvalue at $\x'$, where we note that $(c_1(\x'),c_2(\x'))\neq (0,0)$ by virtue of Lemma \ref{b_12nonvani} (and by $E_{\l}\neq 0$). 

\end{lem}

\begin{proof}
By virtue of $(\ref{Bdef})$, it suffices to prove
\begin{align*}
c_1(a_1b_3^2+a_3b_1^2)+c_2b_1b_2a_3=0,\quad c_1b_1b_2a_3+c_2(a_2b_3^2+a_3b_2^2)=0.
\end{align*}
Using the relation $c_ja_j=b_j$ and Lemma \ref{Gauperp}, we have
\begin{align*}
c_1(a_1b_3^2+a_3b_1^2)+c_2b_1b_2a_3=&b_1\left(b_3^2+(b_1c_1+b_2c_2)a_3\right)\\
=&b_1(b_3^2-a_3b_3c_3)=0.
\end{align*}
The equation $c_1b_1b_2a_3+c_2(a_2b_3^2+a_3b_2^2)=0$ is similarly proved.
\end{proof}

Now Proposition \ref{nullev} immediately follows from Lemma \ref{nullevlem}.


\begin{thebibliography}{99}





%\bibitem{C} J. C. Cuenin, Embedded eigenvalues of generalized Schr\"odinger operators, to appear in J. Spectr. Theor. arXiv:1709.06989.

\bibitem{C2} J. C. Cuenin, Eigenvalue estimates for bilayer graphene. Ann. Henri Poincar\'e 20 (2019), no. 5, 1501--1516.

\bibitem{ES} L. Erd\"os, M. Salmhofer, Decay of the Fourier transform of surfaces with vanishing curvature. Math. Z. 257, (2007), 261--294. 

\bibitem{F} R. L. Frank. Eigenvalue bounds for Schr\"odinger operators with complex potentials. Bull. Lond. Math. Soc., 43(4):745--750, 2011.


%\bibitem{FSa} R. L. Frank, F. Sabin, Spectral cluster bounds for orthonormal systems and oscillatory integral operators in Schatten spaces. Adv. Math. 317 (2017), 157--192.

%\bibitem{FSa2} R. L. Frank, F. Sabin, Restriction theorems for orthonormal functions, Strichartz inequalities, and uniform Sobolev estimates. Amer. J. Math. 139 (2017), no. 6, 1649--1691.


%\bibitem{G} L. Grafakos, Classical Fourier analysis. Second edition. Graduate Texts in Mathematics, 249. Springer, New York, (2008).


\bibitem{IM} I.A. Ikromov, D. M\"uller, Uniform Estimates for the Fourier Transform of Surface Carried Measures in $\re^3$ and an Application to Fourier Restriction, J Fourier Anal Appl 17, (2011), 1292--1332. 


\bibitem{JKL} E. Jeong, Y. Kwon, and S. Lee, Uniform Sobolev inequalities for second order non-elliptic differential operators, Adv. Math. 302 (2016), 323--350.


\bibitem{KY} T. Kato, K. Yajima, Some examples of smooth operators and the associated smoothing effect. Rev. Math. Phys. 1 (1989), no. 4, 481--496.

\bibitem{KT} M. Keel and T. Tao, Endpoint Strichartz estimates, Amer. J. Math. 120 (1998), no. 5, 955--
980.


\bibitem{KM} E. Korotyaev, J. M\o ller, Weighted estimates for the discrete Laplacian on the cubic lattice. Arkiv foer Matematik, Vol. 57, No. 2,(2019), 397--428.






\bibitem{KRS} C. E. Kenig, A. Ruiz, C. D. Sogge, Uniform Sobolev inequalities and unique continuation for second order constant coefficient differential operators. Duke Math. J. 55 (1987), no. 2, 329--347.



%\bibitem{O} R. O'Neil, Convolution operators and L(p,q) spaces. Duke Math. J. 30 1963 129--142.

\bibitem{RS} M. Reed, B. Simon, {\it The Methods of Modern Mathematical 
Physics}, Vol.\ I--IV.  Academic Press, 1972--1980. 


\bibitem{SK} A. Stefanov, P. G. Kevrekidis, Asymptotic behaviour of small solutions for the discrete nonlinear Schr\"odinger and Klein-Gordon equations. Nonlinearity 18 (2005), no. 4, 1841--1857.

\bibitem{S} E. M. Stein. Harmonic analysis: real-variable methods, orthogonality, and oscillatory integrals, volume 43 of Princeton Mathematical Series. Princeton University Press, Princeton, NJ, 1993. With the assistance of Timothy S. Murphy, Monographs in Harmonic Analysis, III.

\bibitem{TT} Y. Tadano, K. Taira, Uniform bounds of discrete Birman-Schwinger operators.  Trans. Amer. Math. Soc. 372 (2019), 5243--5262.

\bibitem{T} K. Taira, Limiting absorption principle on $L^p$-spaces and scattering theory, preprint, arXiv:1904.00505.

\bibitem{V} A. N. Var\u{c}enko, Newton polyhedra and estimates of oscillatory integrals, Funkcional. Anal. i Prilo\u{z}en. 10 (1976), no. 3, 13--38. MR0422257.












\end{thebibliography}
\end{document}